\documentclass[hidelinks,onefignum,onetabnum]{siamart220329}

\usepackage{todonotes}
\usepackage{mathdots}
\usepackage{amssymb}

\def\bfe{{\bf e}}
\def\bfx{{\bf x}}

\usepackage{lipsum}
\usepackage{amsfonts}
\usepackage{graphicx}
\usepackage{epstopdf}
\usepackage{algorithm,algpseudocode}
\ifpdf
  \DeclareGraphicsExtensions{.eps,.pdf,.png,.jpg}
\else
  \DeclareGraphicsExtensions{.eps}
\fi

\newsiamremark{remark}{Remark}
\newsiamremark{hypothesis}{Hypothesis}
\crefname{hypothesis}{Hypothesis}{Hypotheses}
\newsiamthm{claim}{Claim}

\title{Newtonian potentials of Legendre polynomials on rectangles have displacement structure}

\author{S. Olver\thanks{Department of Mathematics, Imperial College, London, England (s.olver@imperial.ac.uk).}}

\usepackage{amsopn}

\def\addtab#1={#1\;&=}

\def\meeq#1{\def\ccr{\\\addtab}
 \begin{align*}
 \addtab#1
 \end{align*}
  }  
  
  \def\leqaddtab#1\leq{#1\;&\leq}

\def\vc#1{\mbox{\boldmath$#1$\unboldmath}}

\def\pr(#1){\left({#1}\right)}
\def\br[#1]{\left[{#1}\right]}
\def\fbr[#1]{\!\left[{#1}\right]}

\def\ip<#1>{\left\langle{#1}\right\rangle}
\def\iip<#1>{\left\langle\!\langle{#1}\right\rangle\!\rangle}

\def\fpr(#1){\!\pr({#1})}

\def\Re{{\rm Re}\,}
\def\Im{{\rm Im}\,}

\def\mapengine#1,#2.{\mapfunction{#1}\ifx\void#2\else\mapengine #2.\fi }

\def\map[#1]{\mapengine #1,\void.}

\def\mapenginesep_#1#2,#3.{\mapfunction{#2}\ifx\void#3\else#1\mapengine #3.\fi }

\def\mapsep_#1[#2]{\mapenginesep_{#1}#2,\void.}

\def\vcbr[#1]{\pr(#1)}

\def\bvect[#1,#2]{
{
\def\dots{\cdots}
\def\mapfunction##1{\ | \  ##1}
\begin{pmatrix}
		 \,#1\map[#2]\,
\end{pmatrix}
}
}

\def\vect[#1]{
{\def\dots{\ldots}
	\vcbr[{#1}]
}}

\def\vectt[#1]{
{\def\dots{\ldots}
	\vect[{#1}]^{\top}
}}

\def\Vectt[#1]{
{
\def\mapfunction##1{##1 \cr} 
\def\dots{\vdots}
	\begin{pmatrix}
		\map[#1]
	\end{pmatrix}
}}

\def\C{{\mathbb C}}

\def\I{{\rm i}}
\def\D{{\rm d}}
\def\dx{\D x}
\def\ds{\D s}
\def\dt{\D t}

\def\tF_#1{{\tt F}_{#1}}

\def\tFC_#1{{\tt T}_{#1}}

\def\secref#1{Section~\ref{Section:#1}}

\def\lmref#1{Lemma~\ref{Lemma:#1}}
\def\propref#1{Proposition~\ref{Proposition:#1}}

\def\thref#1{Theorem~\ref{Theorem:#1}}

\def\corref#1{Corollary~\ref{Corollary:#1}}
\def\figref#1{Figure~\ref{Figure:#1}}

\def\elllRpz_#1{\ell_{#1{\rm z}}^{(\lambda,R),p}}

\def\bbR{{\mathbb R}}
\def\bbC{{\mathbb C}}

\begin{document}

\maketitle

\begin{abstract}
  Particular solutions of the Poisson equation can be constructed via Newtonian potentials, integrals involving the corresponding Green's function which in two-dimensions has a logarithmic singularity. The singularity  represents a significant challenge for computing the integrals, which is typically overcome via specially designed quadrature methods involving a large number of evaluations of the function and kernel. We present an attractive alternative: we show that Newtonian potentials (and their gradient) applied to (tensor products of) Legendre polynomials can be expressed in terms of complex integrals which satisfy simple and explicit recurrences that can be utilised to {\it exactly} compute singular integrals, {\it i.e.}, singular integral quadrature is completely avoided. The inhomogeneous part of the recurrence has low rank structure (its rank is at most three for the Newtonian potential) and hence these recurrences have displacement structure. Using the recurrence directly is a fast approach for evaluation on or near the integration domain that remains accurate for low degree polynomial approximations, while high-precision arithmetic  allows accurate use of the approach for moderate degree polynomials. 
\end{abstract}



\section{Introduction}
In this paper we consider the  {\it Newtonian  potential} and its gradient over the unit square: 
\begin{align*}
\iint_\Omega \log \| \vc x - \vc t \|  f(\vc t) \D \vc t,  \qquad \iint_\Omega  \nabla \log \| \vc x - \vc t \|  f(\vc t) \D \vc t
\end{align*}
where  $\Omega := \{\vectt[x,y] : -1 \leq x,y \leq 1\}$.  Newtonian potentials over rectangles and parallelograms can be expressed explicitly in terms of ${\cal L}$ via an affine change of variables.
We allow $\vc x$ to be anywhere in $\bbR^2$, including $\vc x \in \Omega$. When it is far from $\Omega$ the kernel is smooth and traditional quadrature techniques are effective. However, as $\vc x$ approaches $\Omega$ the kernels become increasingly singular, with weak singularities when  $\vc x \in \Omega$.

Computing such integrals has a long history with many effective quadrature schemes, for example \cite{atkinson1985numerical,greengard1996direct,barnett2014evaluation,af2018adaptive,greengard2021fast,shen2024rapid,anderson2024fast}. We refer the reader for \cite{shen2024rapid} for a recent overview of existing methodology and a detailed description of their limitations for computing the integrals for $\bfx$ near or on $\Omega$, due to the nearly singular or singular integrals. In \cite{shen2024rapid} an approach is advocated for   computing the integrals via expansions in a monomial basis, whose conditioning can be controlled up to degree 20.  In this paper  we investigate a closely related alternative: expand $f$ in Legendre polynomials
\[
f(x,y)  \approx \sum_{k=0}^m \sum_{j = 0}^n f_{kj} \underbrace{P_k(x) P_j(y)}_{P_{kj}(x,y)}
\]
and compute the integrals
\[
\iint_\Omega \log \| \vc x - \vc t \|  P_{kj}(\vc t) \D \vc t,  \qquad \iint_\Omega  \nabla \log \| \vc x - \vc t \|  P_{kj}(\vc t) \D \vc t
\]
in closed form, which will be  related to  a simple recurrence relationship. This recurrence relationship can be solved extremely fast, albeit with some issues with ill-conditioning for large $k$ or $j$.

To derive the recurrence relationship we  express the Newtonian potential as the real part of a complex logarithmic integral
\[
{\cal L} f(z) :=  \iint_\Omega \log (z - (s + \I t)) f(s, t) \, \ds\, \dt
\]
and its gradient in terms of the real and imaginary part of the Stieltjes integral
\[
{\cal S} f(z) := \iint_\Omega {f(s, t) \over z - (s + \I t)} \, \ds\, \dt.
\]
Thus particular solutions to the Poisson equation with the (real-valued) right-hand side $f(x,y)$ are given by
\[
\iint_\Omega \log \| \vc x - \vc t \|  f(\vc t) \D \vc t = \Re {\cal L} f(x+\I y) \approx \sum_{k=0}^m \sum_{j = 0}^n f_{kj}  \Re  L_{kj}(x+ \I y).
\]
where $L_{kj}(z) := {\cal L} P_{kj}(z)$ and we identify $\bfx = \vectt[x,y]$ with $z = x+{\rm i} y$. The gradient can be related to the Stietljes integral via:
\[
\iint_\Omega \nabla \log \| \vc x - \vc t \|  f(\vc t) \D \vc t = \begin{pmatrix} \Re \\ -\Im  \end{pmatrix} {\cal S} f(x+\I y) \approx \sum_{k=0}^m \sum_{j = 0}^n f_{kj}  \begin{pmatrix} \Re \\ -\Im  \end{pmatrix}  S_{kj}(x+ \I y)
\]
where $S_{kj}(z) := {\cal S} P_{kj}(z)$.  The key observation of this article is that $\{L_{kj}(z) \}$ and $\{S_{kj}(z)\}$ both solve simple recurrence relationships that can be use to  compute them {\it exactly}.

 In particular we show the complex logarithmic and Stieltjes integrals of Legendre polynomials have {\it displacement structure}. Define the (infinite) matrices
\[
L(z) := \begin{pmatrix} L_{00}(z) & L_{10}(z) & \cdots \\
			L_{10}(z) & L_{11}(z) & \cdots \\
			\vdots&\vdots & \ddots
			\end{pmatrix}, \qquad S(z) := \begin{pmatrix} S_{00}(z) & S_{10}(z) & \cdots \\
			S_{10}(z) & S_{11}(z) & \cdots \\
			\vdots&\vdots & \ddots
			\end{pmatrix}.
\]
 We show in \thref{LogDisplacement} that $L(z)$ satisfies a Sylvester equation of the form
 \[
C L(z) - \I L(z) C^\top  = F(z)
\]
where $F(z)$ has at most rank $3$ and $C$ is tridiagonal. 
We further show in \lmref{StieltjesDisplacement} that $S(z)$ satisfies a Sylvester equation of the form
\[
(zI - B ) S(z) - {\I } S(z) B^\top  = 4\bfe_0 \bfe_0^\top,
\]
that is the right-hand side has rank $1$ and $B$ is also tridiagonal.
Displacement structure is a powerful property with substantial numerical implications, see \cite{beckermann2017singular,beckermann2019bounds,heinig1984algebraic,ballew2025akhiezer}.

The paper is structured as follows: 

\noindent{\it \secref{intervals}}: The methodology we develop parallels singular integrals over intervals, in particular it utilises a simple relationship between complex logarithmic and Stieltjes integrals. Numerical solution of Stieltjes integrals via recurrence relationships has a long history going back to Gautschi \cite{gautschi1987computing}, see also the review in \cite{SOActa}. In this section we give a brief review of these results and extend to a modified complex logarithmic integral over a complex interval needed in the derivation of the results on a square.

\noindent{\it \secref{squares}}: We show that the Stieltjes integrals $S_{kj}(z)$ satisfy a simple ``5-point stencil'' recurrence relationship by adapting the technique from intervals. By mimicking the connection between complex logarithmic and Stieltjes integrals used on intervals, we show that $L_{kj}(z)$ simultaneously satisfies two ``5-point stencil''  recurrence relationships with an explicitly constructed inhomogeneous term. Moreover, we can construct the boundary cases $S_{k0}(z)$, $S_{0j}(z)$, $L_{k0}(z)$, and $L_{0j}(z)$ directly.

\noindent{\it \secref{Sylvester}}:  The recurrences can be recast as Sylvester's equations and the recurrence relationships for $L(z)$ can be combined in a way that exposes the displacement structure in $L(z)$ and $S(z)$. 

\noindent{\it \secref{computation}}:   An implication of these results is that we can directly  construct $L_{kj}(z)$ and $S_{kj}(z)$ by direct substitution in the recurrence relationship. This proves to be an effective and fast approach to construction that is exact. However, in practice there are  errors due to round-off error so it may only be effective for $z$ near the square and for low degree polynomials. High-precision arithmetic mitigates this issue,  allowing for its effective usage for moderate degree polynomials.

\noindent{\it \secref{conclusion}}: We conclude by discussing the implications of displacement structure to both computation of Newtonian potentials and their gradient, as well as the potential for generalising to other kernels, geometries, and higher-dimensions.

\begin{remark}
An experimental implementation of the results are available in the Julia package MultivariateSingularIntegrals.jl  \cite{MultivariateSingularIntegrals}.
\end{remark}

\section{Log and Stieltjes integrals on intervals}\label{Section:intervals}

A classic result (cf.~\cite{gautschi1987computing}) is that Stieltjes integrals of weighted orthogonal polynomials satisfy the same three-term recurrence as the orthogonal polynomials themselves, apart from a simple inhomogeneity.  Defining
\[
{\cal S}_{[-1,1]} f(z) := \int_{-1}^1  {f(t) \over z - t} \dt
\]
and $S_k(z) := {\cal S}_{[-1,1]} P_k(z)$, the three-term recurrence for Legendre polynomials (assuming $P_{-1}(x) = 0$)
\meeq{
x P_k(x) = {k \over 2k+1} P_{k-1}(x) +  {k+1 \over 2k+1} P_{k+1}(x)
}
tells us the three-term recurrence for their Stieltjes integrals:
\meeq{
z S_k(z) = \int_{-1}^1 {z - t \over z- t} P_k(t) \dt + \int_{-1}^1 {t P_k(t) \over z- t} \dt \cr
&\qquad = 2 \delta_{k0} + {k \over 2k+1} S_{k-1}(x) +  {k+1 \over 2k+1} S_{k+1}(x).
}
Note we also know the initial condition that can be used to kick-off the recurrence,
\[
S_0(z) = \int_{-1}^1 {\dt \over z - t}  = \log(z+1) - \log(z-1),
\]
thus we can deduce $S_k(z)$ by forward-substitution with the recurrence. 

Now consider the complex logarithmic integral
\[
{\cal L}_{[-1,1]} f(z) := \int_{-1}^1 \log(z-t) f(t) \dt
\]
and define $L_k(z) := {\cal L}_{[-1,1]} P_k(z)$. We first connect the complex logarithmic integral to the Stieltjes integral of a specific ultraspherical polynomial \cite[Section 18.3]{DLMF}. For $\lambda \neq 0, \lambda > -1/2$, the ultraspherical polynomials $C_n^{(\lambda)}(x)$ are orthogonal with respect to the weight $(1-x^2)^{\lambda-1/2}$ on $[-1,1]$ with the normalisation constant $2^n (\lambda)_n / n!$ where $(\lambda)_n = \Gamma(\lambda+n )/\Gamma(\lambda)$ is the Pochhammer symbol. The definition can be extended to $\lambda = -1/2$, where the weight is no longer integrables by an explicit expression in terms of Hypergeometric functions \cite[18.5.9]{DLMF}, or more explicitly by the formula
\[
 C_{k}^{(-1/2)}(x) =  \begin{cases} 
 1 & k = 0 \\
-x & k = 1 \\
 {(1-x^2) C_{k-2}^{(3/2)}(x) \over k (k-1) } & \hbox{otherwise}
 \end{cases}.
\]
We will denote the Stieltjes integral of ultraspherical polynomials as
\[
S_k^{(\lambda)}(z) := \int_{-1}^1 {C_{k}^{(\lambda)}(t) \over z - t} \dt.
\]

\begin{proposition}\label{Proposition:log2cauchy}
\meeq{
L_k(z) = - S_{k+1}^{(-1/2)}(z) +  \begin{cases} \log (z+1) + \log(z-1) & k = 0 \\ 0 & \hbox{\rm otherwise} \end{cases}.
}
\end{proposition}
\begin{proof}
Integrating by parts allows us to relate the complex logarithmic integral to the Stieltjes integral which is more amenable to computation as:
\[
{\cal L}_{[-1,1]} f(z) = \int_{-1}^1 f(t) \dt \log(z+1) - \int_{-1}^1 {F(t) \over z - t} \dt
\]
where $F(x) = \int_x^1 f(s)\,\ds$.  In the case of Legendre polynomials its indefinite integral is given in terms of certain ultraspherical polynomials \cite[18.9.19]{DLMF}: for $k > 0$,
\[
  \int_x^1 P_k(t)\dt = C_{k+1}^{(-1/2)}(x) = {(1-x^2) C_{k-1}^{(3/2)}(x) \over k (k+1) }.
\]
Thus, since $\int_{-1}^1 P_k(t) \dt = 0$ when $k \neq 0$, we have 
\meeq{
L_k(z) = - \int_{-1}^1 { C_{k+1}^{(-1/2)}(t) \over z - t} \dt = -S_{k+1}^{(-1/2)}(z).
}
For $k = 0$ we have, since $C_1^{(-1/2)}(x) = -x$, 
\[
  \int_x^1 P_0(t)\dt = 1 - x = 1 + C_1^{(-1/2)}(x)
\]
thus we also have
\[
L_0(z) = 2\log(z+1)- S_1^{(-1/2)}(z) - \int_{-1}^1  {1 \over z- t} \dt,
\]
which simplifies to the desired expression.

\end{proof}

This can be used to deduce a simple recurrence relationship:

\begin{theorem}\label{Theorem:legendre_recurrence}

The complex logarithmic integral of Legendre polynomials satisfies a three-term recurrence:
\meeq{
z L_k(z) = {k - 1 \over 2k + 1} L_{k-1}(z) + {k + 2 \over 2k + 1} L_{k+1}(z) + \lambda_k(z)
}
for
\[
\lambda_k(z) := \begin{cases} (z-1)\log(z-1) + (z+1)\log(z+1) & k =0 \\
-2/3 & k = 1\\
0 & \hbox{otherwise}
\end{cases}.
\]
\end{theorem}

\begin{proof}
We shall deduce the recurrence relationship for $L_k(z)$ using a technique also advocated in \cite{SOActa,slevinsky2017fast}. We employ the recurrence relationship for ultraspherical polynomials
\begin{equation}\label{eq:ultraspherical_three_term}
  t C_{k}^{(\lambda)}(t) = {k + 2\lambda -1 \over 2(k +\lambda)} C_{k-1}^{(\lambda)}(t) + {k+1 \over 2(k + \lambda)} C_{k+1}^{(\lambda)}(t).
\end{equation}
which remains valid for $\lambda = -1/2$. Note we have 
\meeq{
\int_{-1}^1 C_k^{(-1/2)}(x) \dx =\int_{-1}^1 \begin{cases} 1 & k = 0 \\ -x & k = 1 \\{ (1-x^2) C_{k-2}^{(3/2)}(x) \over (k-1) k} & \hbox{otherwise} \end{cases} \dx = \begin{cases}
2 & k = 0\\
2/3 & k =2 \\
0 &\hbox{otherwise}
\end{cases}.
}
It follows that
\begin{align*}
    z L_k(z) &=   - \int_{-1}^1{z - t \over z - t} C_{k+1}^{(-1/2)}(t) \dt -   \int_{-1}^1{t \over z - t} C_{k+1}^{(-1/2)}(t) \dt  \\
    &\qquad + \begin{cases} z \log(z+1) + z \log(z-1) & k = 0 \\ 0 & \hbox{otherwise} \end{cases}\\
    &=   -   \int_{-1}^1{1 \over z - t} \pr({ {k -1 \over 2k +1} C_k^{(-1/2)}(t) + {k+2 \over 2k +1} C_{k+2}^{(-1/2)}(t)})\dt\\
    &\qquad + \begin{cases} z\log(z+1) + z\log(z-1) & k = 0 \\ -2/3 & k = 1 \\ 0 & \hbox{otherwise} \end{cases} \\
    &=   {k - 1 \over 2k + 1} L_{k-1}(z) + {k + 2 \over 2k + 1} L_{k+1}(z). \\
    &\qquad + \begin{cases} \int_{-1}^1 {1 \over z - t} \D t +  z\log(z+1) + z\log(z-1) & k = 0 \\ -2/3 & k = 1 \\ 0 & \hbox{otherwise} \end{cases},
\end{align*}
which simplifies to the desired result.

\end{proof}

To extend the results to 2D we will need to work with the following variant where we complexify the integration variable:
\[
{\cal M}_{[-1,1]} f(z) := \int_{-1}^1 \log(z- \I t) f(t) \dt
\]
and $M_k(z) := {\cal M}_{[-1,1]} P_k(z)$.
This is a subtle change: writing $z = x + \I y$, for $-1 \leq y \leq 1$ and $x < 0$ the integral passes over the branch cut of the logarithm. We nevertheless can relate this to ${\cal L}$:

\begin{proposition} For $z = x + \I y$ we have
\[
{\cal M}_{[-1,1]} f(z) = {\cal L}_{[-1,1]} f(-\I z)  + {\I \pi \over 2} \begin{cases}  
\int_{-1}^1 f(t) \dt & x > 0 \hbox{\rm\ or } y \geq 1   \\
								-3  \int_{-1}^1 f(t) \dt  & x \leq 0 \hbox{\rm\ and } y \leq -1 \\
								\left(\int_{-1}^y - 3 \int_y^1 \right) f(t) \dt & x < 0 \hbox{\rm\ and } -1 < y < 1
\end{cases}
\]
\end{proposition}
\begin{proof}
We use the formula
\[
\log(\I z) =  \log|z| + \I \arg(\I z) = \log z + {\I \pi \over 2}\begin{cases}
1 & -\pi < \arg z \leq \pi/2 \\ 
-3 &\hbox{otherwise}
\end{cases}
\]
to deduce
\meeq{
\int_{-1}^1 f(t) \log(z-\I t) \dt =  \int_{-1}^1 f(t) \log(-\I z- t) \dt \\
&\qquad+ {\I  \pi \over 2} \int_{-1}^1 f(t) \begin{cases}
1 & -\pi < \arg(-\I z - t) \leq \pi/2 \\ 
-3 &\hbox{otherwise}
\end{cases} \dt
}
The result follows by considering the region where $-\I z - t = y-t-\I x$ is under the three different cases:
\begin{enumerate}
\item If $x > 0$ then $\Im(y-t-\I x) = -x < 0$ and hence  $- \pi < \arg(-\I z - t) < 0$, whilst if $y \geq 1$ we have  $\Re(y-t-\I x) \geq y-1 \geq 0$ and  $0 \leq \arg(-\I z - t) \leq \pi/2$.
\item If $x \leq 0 \hbox{\rm\ and } y < -1$ then $\Im(y-t-\I x) = -x  \geq 0$ and $\Re(y-t-\I x) \leq y+1  < 0$ and hence  $\pi/2 < \arg(-\I z - t) \leq \pi$. If $y = -1$ the integrand only differs at a single point so its value is unchanged.
\item If $x \leq 0 \hbox{\rm\ and } -1 < y < 1$  then for $t < y$ we have $\Re(y-t-\I x)> 0$ whilst for $t > y$ we have $\Re(y-t-\I x)< 0$ and hence the integrand changes value at the point $y=t$. 
\end{enumerate}
\end{proof}

Specialising to Legendre polynomials we have the following:

\begin{theorem}\label{Theorem:MasL}
For $z = x + \I y$ we have
\meeq{
M_k(z) = L_k(-\I z) + \I \pi  \delta_{k0} \begin{cases}  
1 &  x > 0 \hbox{\rm\ or } y \geq 1  \\
								-3 & x \leq 0 \hbox{\rm\ and } y \leq -1 \\
								-1	&  x < 0 \hbox{\rm\ and } -1 < y < 1 \\
								\end{cases}\cr
								&\qquad  -  2\I \pi \begin{cases}  
 C_{k+1}^{(-1/2)}(y) & x < 0 \hbox{\rm\ and } -1 < y < 1 \\
0 & \hbox{\rm\ otherwise}
								\end{cases}.
}

\end{theorem}

%

We can deduce a three-term recurrence for $M_k(z)$ as well:

\begin{corollary}\label{Corollary:Mrecurrence}
For $z = x + \I y$ we have	
\meeq{
z M_k(z) = \I {k-1 \over 2k+1} M_{k-1}(z) + \I {k+2 \over 2k+1}  M_{k+1}(z) + \mu_k(z)
}
for
\begin{align*}
\mu_k(z) &:= \begin{cases}
(z-\I) \log(z-\I) + (z+\I) \log_-(z+\I) & k = 0 \\
-2\I / 3 & k = 1 \\
0 & \hbox{otherwise}
\end{cases}  \\
& \qquad - \begin{cases} 2 \I \pi x C_{k+1}^{(-1/2)}(y) & \hbox{$x < 0$\hbox{ and } $-1 < y < 1$} \\
0 & \hbox{otherwise}  \end{cases},
\end{align*}
where we define the logarithm with the other limit along its branch cut:
\[
\log_- z := \begin{cases}
\log|z| - \I \pi  & y = 0 \hbox{ and } x < 0 \\
\log z & \hbox{otherwise}
\end{cases}.
\]
\end{corollary}
\begin{proof}

Considering the different cases we find that:
\meeq{
(z - \I) \log(z - \I) + (z + \I) \log_-(z + \I)  = 
(z - \I) \log(-\I z - 1) + (z + \I) \log(-\I z + 1) \\
&\qquad +  \pi \begin{cases}
\I  z & \hbox{$x > 0$ or $y \geq 1$ } \\
  -3 \I  z & \hbox{$x \leq 0$ and $y \leq -1$ } \\
  -\I z -2 & \hbox{$x < 0$ and $-1 < y < 1$ } 
 \end{cases}.
}
 For the $k=0$ case we thus find, using \thref{MasL},
\meeq{
z M_0(z) =  \I (-\I z L_0(-\I z)) + \I \pi  z \begin{cases}  
1 &  x > 0 \hbox{\rm\ or } y \geq 1   \\
								-3 &  x \leq 0 \hbox{\rm\ and } y \leq -1 \\
								2 y -1	&  x < 0 \hbox{\rm\ and } -1 < y < 1
								\end{cases}\ccr
  = 2\I L_1(-\I z)  +(z - \I) \log(-\I z - 1) + (z + \I) \log(-\I z + 1)\\
  &\qquad  + \I \pi  z \begin{cases}  
1 &  x > 0 \hbox{\rm\ or } y \geq 1   \\
								-3 &  x \leq 0 \hbox{\rm\ and } y \leq -1 \\
								2 y -1	&  x < 0 \hbox{\rm\ and } -1 < y < 1 \
								\end{cases} \ccr
  = 2 \I M_1(z) + \pi \begin{cases}  
2 y  \I   z+2 - 4 C_2^{(-1/2)}(y) & x < 0 \hbox{\rm\ and } -1 < y < 1 \\
0 & \hbox{\rm\ otherwise}
								\end{cases}  \\
								&\qquad   +(z - \I) \log(z - \I) + (z + \I) \log(z + \I).
}
This reduces to the result when we observe
\[
2y \I z + 2 -4  C_2^{(-1/2)}(y) =  2x y \I - 2y^2 + 2 -2 (1-y^2) = 2x y \I.
\]
 For $k>0$  we have, using the three-term recurrence to expand $y C_{k+1}^{(-1/2)}(y)$,
 \meeq{
z M_k(z) =  \I (-\I z L_k(-\I z)) -2 \I \pi  z \begin{cases}  
 C_{k+1}^{(-1/2)}(y) & x < 0 \hbox{\rm\ and } -1 < y < 1 \\
0 & \hbox{\rm\ otherwise}
								\end{cases}  \ccr
  = \I {k-1 \over 2k+1} L_{k-1}(-\I z) + \I {k+2 \over 2k+1} L_{k+1}(-\I z)  -  2\I/3 \delta_{k1} \\
  &\qquad   + 2 \pi  \begin{cases}  
{k -1 \over 2k+1} C_{k}^{(-1/2)}(y) + {k +2 \over 2k+1} C_{k+2}^{(-1/2)}(y) & x < 0 \hbox{\rm\ and } -1 < y < 1 \\
0 & \hbox{\rm\ otherwise}
								\end{cases}\\
&\qquad   -2 \I \pi  x \begin{cases}  
 C_{k+1}^{(-1/2)}(y) & x < 0 \hbox{\rm\ and } -1 < y < 1 \\
0 & \hbox{\rm\ otherwise}
								\end{cases}
}
which reduces to the result via  \thref{MasL}.


\end{proof}

\begin{remark}
The proceeding theorem may not apply for  complex floating point numbers due to the existence of a negative 0: since $-\I (0.0 + \I y) = y - 0.0\I$ some implementations of the logarithm may be evaluate as if it is below the branch cut. This can be avoided by specialising the implementation of the logarithm for complex floating point numbers.
\end{remark}

\section{Complex logarithmic and Stieltjes integrals on squares}\label{Section:squares}

We now turn to deducing recurrence relationships for the two-dimensional integrals $S_{kj}(z)$ and $L_{kj}(z)$.

\subsection{Stieltjes integrals} We begin with deriving a recurrence relationship for the 2D Stieltjes integral
\[
S_{kj}(z) := \iint_\Omega {P_k(s) P_j(t) \over z -(s+ \I t)} \ds \dt.
\]

\begin{proposition}
The Stieltjes integral $S_{kj}(z)$ satisfies the ``5-point stencil'' recurrence relationships:
\meeq{
z S_{kj}(z) =  \I {j \over 2j+1} S_{k,j-1}(z) + \I {j+1 \over 2j+1} S_{k,j+1}(z)\cr
 &\qquad 
 + {k \over 2k+1} S_{k-1,j}(z) + {k+1 \over 2k+1} S_{k+1,j}(z) +4 \delta_{j0} \delta_{k0}.
}
\end{proposition}
\begin{proof}
This follows from the three-term recurrence for $S_k(z)$. That is:
\meeq{
z S_{kj}(z) = \int_{-1}^1   P_k(s)  (-\I z) \int_{-1}^1 {P_j(t) \over -\I z +  \I s - t} \dt \ds \ccr
= \int_{-1}^1 P_k(s) (-\I z + \I s)   S_j(-\I z + \I s) \ds - \I \int_{-1}^1 s P_k(s)  S_j(-\I z + \I s) \ds \ccr
= {j \over 2j+1} \int_{-1}^1 P_k(s)  S_{j-1}(-\I z + \I s) \ds +  {j +1\over 2j+1} \int_{-1}^1 P_k(s)  S_{j+1}(-\I z + \I s) \ds  \\
&\qquad+ 2 \delta_{j0} \int_{-1}^1 P_k(s) \ds  - \I {k \over 2k+1} \int_{-1}^1 P_{k-1}(s)  S_j(-\I z + \I s) \ds  \\
&\qquad- \I {k+1 \over 2k+1} \int_{-1}^1  P_{k+1}(s)  S_j(-\I z + \I s) \ds \ccr
 =  \I {j \over 2j+1} S_{k,j-1}(z) + \I {j+1 \over 2j+1} S_{k,j+1}(z)\cr
 &\qquad 
 + {k \over 2k+1} S_{k-1,j}(z) + {k+1 \over 2k+1} S_{k+1,j}(z) +4 \delta_{j0} \delta_{k0}.
}
\end{proof}
To use this recurrence in practice we need to know initial conditions. It turns out we can deduce these directly in terms of the one-dimensional singular integrals, which are deducible directly from recurrences:

\begin{proposition} \label{Proposition:StieltjesRowCol}
\meeq{
S_{k0}(z) = (-1)^k \I (  M_k(-1-\I z) -  M_k(1-\I z)),\ccr
S_{0j}(z) = M_j(z+1) -  M_j(z-1)
}

\end{proposition}

\begin{proof}
We use the fact that we have an explicit form for $S_0(z)$:
\meeq{
S_{k0}(z) = -\I \int_{-1}^1 P_k(s) S_0(-\I z + \I s) \ds \ccr
 = -\I \int_{-1}^1 P_k(s) \br[\log(1-\I z + \I s) - \log(-\I z + \I s - 1)] \ds \ccr
  = \I (-1)^k \int_{-1}^1 P_k(s) \br[\log(-1-\I z - \I s) - \log(1-\I z - \I s )] \ds \ccr
  = (-1)^k \I (  M_k(-1-\I z) -  M_k(1-\I z)).
}
Similarly,
\meeq{
S_{0j}(z) =  \int_{-1}^1 P_j(t) S_0(z - \I t) \dt  \ccr
=  \int_{-1}^1 P_j(t) \br[\log(1+ z-\I t) - \log(z-\I t -1)] \ds \ccr
 = M_j(z+1) -  M_j(z-1).
}

\end{proof}

\subsection{Complex logarithmic integral}

The key idea in the one-dimensional complex logarithmic integrals  was their reduction to a Stieltjes integral which lead to a  recurrence relationship. Here we observe that one can deploy this same technique on two-dimensional complex logarithmic integrals.

\begin{theorem}\label{Theorem:2d_legendre_recurrence}
  The complex logarithmic integral of  Legendre polynomials satisfies two ``5-point stencil"-like recurrence relations:
\meeq{%
z L_{kj} =   \frac{k-1}{2k+1}L_{k-1,j}  + \frac{k+2}{2k+1}L_{k+1,j} + \I \left( \frac{j}{2j+1}L_{k,j-1} + \frac{j+1}{2j+1} L_{k,j+1} \right) + F_{kj}^{(1)}(z), \ccr
=\frac{k}{2k+1}L_{k-1,j} +  \frac{k+1}{2k+1} L_{k+1,j}  + \I \left( \frac{j-1}{2j+1} L_{k,j-1} +  \frac{j+2}{2j+1} L_{k,j+1}\right) + F_{kj}^{(2)}(z),
}
where we take $L_{-1,j} = L_{k,-1} = 0$ and have
\meeq{
F_{0j}^{(1)}(z) =    \I {M_{j+1}(z-1) + M_{j+1}(z+1) - M_{j-1}(z-1) - M_{j-1}(z+1)  \over 2j+1} \cr
&\qquad + \mu_j(z-1) + \mu_j(z+1),  \ccr
F_{10}^{(1)}(z) = -4/3, \ccr
F_{kj}^{(1)}(z) = 0\qquad\hbox{ for $k \neq 0$ and $(k,j) \neq (1,0)$},\ccr
F_{k0}^{(2)}(z) = {L_{k+1}(z-\I) + L_{k+1}(z+\I) - L_{k-1}(z-\I) - L_{k-1}(z+\I) \over 2k+1}\cr
&\qquad  + \lambda_k(z-\I) + \lambda_k(z+\I) + \beta_{k0}(z) \ccr
F_{01}^{(2)}(z) = - 4\I/3 + \beta_{01}(z), \ccr
 F_{kj}^{(2)}(z) =  \beta_{kj}(z)\qquad\hbox{ for $j \neq 0$ and $(k,j) \neq (0,1)$}
}
where
\meeq{
\beta_{kj}(z) = 2\I\pi C_{j+1}^{(-1/2)}(y) \begin{cases} C_{k+2}^{(-3/2)}(x)/3 - \delta_{k0} x + \delta_{k1}/3 & z \in \Omega \\
-2x \delta_{k0} +2\delta_{k1}/3 & x < -1 \hbox{ and } -1 < y < 1 \\
0 &\hbox{otherwise}
\end{cases}.
}
\end{theorem}
\begin{proof}

For the first recurrence relationship we roughly follow the approach of the two-dimensional Stieltjes integral using the recurrence in \thref{legendre_recurrence}: 
\meeq{
z L_{kj}(z) =  \int_{-1}^1 P_j(t)  z \int_{-1}^1 P_k(s) \log(z - (s + \I t)) \ds \dt   \ccr
= \int_{-1}^1 P_j(t) (z-\I t) L_k(z- \I t) \dt + \I \int_{-1}^1 t P_j(t) L_k(z- \I t) \dt \ccr
= {k-1 \over 2k+1} \int_{-1}^1 P_j(t)L_{k-1}(z- \I t) \dt + {k+2 \over 2k+1} \int_{-1}^1 P_j(t)L_{k+1}(z- \I t) \dt \\
&\qquad + \I  {j \over 2j+1} \int_{-1}^1 P_{j-1}(t) L_k(z- \I t) \dt  + \I  {j+1 \over 2j+1} \int_{-1}^1 P_{j+1}(t) L_k(z- \I t) \dt  \\
& \qquad+ \int_{-1}^1 P_j(t)  \lambda_k(z-\I t) \dt.
}
Using \corref{Mrecurrence} we find
\begin{align*}
\int_{-1}^1 &P_j(t) (z-\I t) \log(z - \I t) \dt   = z M_j(z) -  \I \int_{-1}^1 t P_j(t) \log(z - \I t) \ds\ccr
 = \I {j-1 \over 2j+1} M_{j-1}(z) + \I {j+2 \over 2j+1} M_{j+1}(z) \\
 &\qquad  - \I {j \over 2j+1} M_{j-1}(z) - \I {j+1 \over 2j+1} M_{j+1}(z) + \mu_j(z) \ccr
  = -\I {1 \over 2j+1} M_{j-1}(z) + \I {1 \over 2j+1} M_{j+1}(z) + \mu_j(z).
\end{align*}
Hence we can simplify:
\begin{align*}
 \int_{-1}^1 & P_j(t)  \lambda_k(z-\I t) \dt = \int_{-1}^1 P_j(t)  \begin{cases} {(z-\I t-1)\log(z-\I t-1)  \atop + (z-\I t+1)\log(z-\I t+1)} & k =0 \\
-2/3 & k = 1\\
0 & \hbox{otherwise}
\end{cases} \dt  \ccr
=	\begin{cases} {-\I {1 \over 2j+1} \pr(M_{j-1}(z-1) + M_{j-1}(z+1)) \atop + \I {1 \over 2j+1} (M_{j+1}(z-1) + M_{j+1}(z+1) ) + (\mu_j(z-1) + \mu_j(z+1)) }& k =0 \\
-4/3 & k = 1\hbox{ and } j=0 \\
0 & \hbox{otherwise}
\end{cases}
\end{align*}

For the second recurrence we use \corref{Mrecurrence} to reduce the integral: 
\meeq{
z L_{kj}(z) =  \int_{-1}^1 P_k(s)  z \int_{-1}^1 P_j(t) \log(z - (s + \I t)) \dt \ds  \ccr
=  \int_{-1}^1 P_k(s)  (z-s) M_j(z-s) \ds  + \int_{-1}^1 s P_k(s) M_j(z-s) \ds \ccr
=  \I {j-1 \over 2j+1} \int_{-1}^1 P_k(s)   M_{j-1}(z-s) \ds + \I {j+2 \over 2j+1}  \int_{-1}^1 P_k(s)   M_{j+1}(z-s) \ds  \\
&\quad + {k \over 2k+1} \int_{-1}^1  P_{k-1}(s) M_j(z-s) \ds + {k+1 \over 2k+1} \int_{-1}^1  P_{k+1}(s) M_j(z-s) \ds  \cr
&\quad + \int_{-1}^1 P_k(s) \mu_j(z-s) \ds \ccr
=  \I {j-1 \over 2j+1}L_{k,j-1}(z) + \I {j+2 \over 2j+1}  L_{k,j+1}(z)  \\
&\quad + {k \over 2k+1} L_{k-1,_j}(z) + {k+1 \over 2k+1} L_{k+1,_j}(z)  + \int_{-1}^1 P_k(s) \mu_j(z-s) \ds
 }
where $\mu_j$ is defined in \corref{Mrecurrence}. We thus need only simplify the last integral. Note that
\begin{align*}
\int_{-1}^1 &P_k(s) (z-s) \log(z - s) \ds   = z L_k(z) - \int_{-1}^1 s P_k(s) \log(z - s) \ds\ccr
 = 
{k-1 \over 2k+1} L_{k-1}(z) + {k+2 \over 2k+1} L_{k+1}(z)  -  {k \over 2k+1}L_{k-1}(z) - {k+1 \over 2k+1}L_{k+1}(z) + \lambda_k(z)\ccr
 = 
{-1 \over 2k+1} L_{k-1}(z) + {1 \over 2k+1} L_{k+1}(z) + \lambda_k(z).
\end{align*}
Furthermore, if $-1 < x < 1$ we have
\begin{align*}
\int_{-1}^1  &P_k(s)  \begin{cases}  x-s  & x < s \\ 0 & \hbox{otherwise} \end{cases} \ds= \int_x^1 P_k(s) (x-s) \ds \ccr
= x (C_{k+1}^{(-1/2)}(x)  + \delta_{k0}) - {k \over 2k+1} \int_x^1 P_{k-1}(s) \ds - {k+1 \over 2k+1} \int_x^1 P_{k+1}(s) \ds \ccr
= -{1 \over 2k+1} C_k^{(-1/2)}(x) + {1 \over 2k+1} C_{k+2}^{(-1/2)}(x) + x \delta_{k0} -\delta_{k1}/3.
\end{align*}
If $x < -1$ then
\begin{align*}
\int_{-1}^1  &P_k(s)  \begin{cases}  x-s  & x < s \\ 0 & \hbox{otherwise} \end{cases} \ds= \int_{-1}^1 P_k(s) (x-s) \ds 
= \begin{cases}
2x & k = 0 \\
-2/3 & k = 1 \\
0 & \hbox{otherwise}
\end{cases}.
\end{align*}
These combine to give the desired result.


\end{proof}

To kick off the recurrence we also need an explicit form for $L_{00}$:

\begin{proposition}\label{Proposition:00entry}
\[
L_{00}(z) = (1-z) M_0(z-1) + \I M_1(z-1) + (1+z) M_0(z+1) - \I M_1(z+1) - 4.
\]
\end{proposition}
\begin{proof}
This follows from:
\meeq{
L_{00}(z) = \int_{-1}^1 \int_{-1}^1 \log(z - (s+\I t))  \, \ds\, \dt = \int_{-1}^1 L_0(z- \I t) \, \dt \ccr
= \int_{-1}^1 \left[(1- z+ \I t) \log(z - \I t - 1) + (1 + z - \I t) \log(z - \I t + 1) - 2 \right] \, \dt \ccr
= (1-z) M_0(z-1) + \I M_1(z-1) + (1+z) M_0(z+1) - \I M_1(z+1) - 4.
}
\end{proof}

Having two different 5-point stencils allows us to eliminate degrees of freedom. In particular, we can deduce  inhomogeneous three-term recurrence relationships that give the first row and column:

\begin{corollary}\label{Corollary:Lfirstrowcol}
\meeq{
z L_{k0}(z) = {k-2 \over 2k+1} L_{k-1,0}(z) + {k+3 \over 2k+1} L_{k+1,0}(z) + 2F_{k0}^{(1)}(z)-F_{k0}^{(2)}(z) \ccr
z L_{0j}(z) = \I {j-2 \over 2j+1} L_{0,j-1}(z) + \I {j+3 \over 2j+1} L_{0,j+1}(z)  + 2F_{0j}^{(2)}(z)-F_{0j}^{(1)}(z) 
}
\end{corollary}
\begin{proof}
The first part follows from multiplying the first formula in \thref{2d_legendre_recurrence} by 2 and subtracting the second formula, and the second part follows the other way round.

\end{proof}

\section{Recasting as a Sylvester equation}\label{Section:Sylvester}

We can rewrite the recurrence relationhips as (infinite) Sylvester equations.  The Stieltjes integral for Legendre polynomials has displacement structure with rank 1:

\begin{lemma}\label{Lemma:StieltjesDisplacement}
The (infinite) matrix $S(z)$ satisfies the recurrence
\[
z S(z) = B S(z) + \I S(z)  B^\top + 4 \bfe_0 \bfe_0^\top
\]
for the matrix
\[
B = \begin{pmatrix}
0 & 1 \\ 
1/3 & 0 & 2/3 \\
& 2/5 & 0 & 3/5 \\ 
&& 3/7 & 0 & 4/7 \\
 &&& 4/9 & 0 & \ddots \\
 &&&&\ddots & \ddots
\end{pmatrix}.
\]
\end{lemma}
The Logarithmic integral for Legendre polynomials satisfies two Sylvester equations for
\[
F^{(\ell)}(z) = \begin{pmatrix} F_{00}^{(\ell)}(z) & F_{10}^{(\ell)}(z) & \cdots \\
			F_{10}^{(\ell)}(z) & F_{11}^{(\ell)}(z) & \cdots \\
			\vdots&\vdots & \ddots
			\end{pmatrix}:
\]

\begin{lemma}
\meeq{
z L(z) = AL(z) + \I L(z) B^\top + F^{(1)}(z),\ccr
z L(z) = BL(z) + \I L(z) A^\top + F^{(2)}(z)
}
where $B$ was defined in \lmref{StieltjesDisplacement} and
\meeq{
A = \begin{pmatrix}
0 & 2 \\ 
0 & 0 & 1 \\
& 1/5 & 0 & 4/5 \\ 
&& 2/7 & 0 & 5/7 \\
 &&& 3/9 & 0 & \ddots \\
 &&&&\ddots & \ddots
\end{pmatrix}.
}
\end{lemma}

\begin{remark}
 Note that $B$ is the multication-by-$x$ matrix for Legendre polynomials and $A$ is the principle subsection of the multiplication-by-$x$ matrix for the ultraspherical polynomials $C_k^{(-1/2)}(x)$.
\end{remark}

We can combine these two recurrences to deduce a single recurrence where the entire $z$ dependence is in the right-hand side:

\begin{theorem}\label{Theorem:LogDisplacement}
\[
C L(z) - \I L(z) C^\top  = F(z)
\]
where $F(z) := F^{(2)}(z) - F^{(1)}(z)$ has at most rank 3  and
\[
C := A - B = \begin{pmatrix}
0 & 1 \\
-1/3 & 0 & 1/3 \\
& -1/5 & 0 & 1/5 \\
& & -1/7 & 0 & \ddots \\
&&& \ddots & \ddots
\end{pmatrix}.
\]

\end{theorem}
\begin{proof}
The form of the Sylvester equation follows from subtracting the Sylvester equations in the preceding lemma. The fact that $F(z)$ has rank 3 follows since dropping the first row and column we find, for $z \in \Omega$,
\meeq{
F(z)[2:\infty,2:\infty] = \begin{pmatrix} \beta_{11}(z) & \beta_{12}(z) & \cdots \\ \beta_{21}(z) & \beta_{22}(z) & \cdots \\ \vdots & \vdots & \ddots \end{pmatrix} \ccr
= 
2\I\pi  \begin{pmatrix} C_3^{(-3/2)}(x)+1/3 \\C_4^{(-3/2)}(x)\\C_5^{-3/2}(x) \\ \vdots \end{pmatrix} \begin{pmatrix} C_2^{(-1/2)}(y) & C_3^{(-1/2)}(y)  & \cdots \end{pmatrix}
}
which has rank 1 (and similarly for $z \notin \Omega$).
\end{proof}

\section{Computation via recurrences} \label{Section:computation}

The derived recurrence relationships combined with the fact that we know the first row and column of $L(z)$ and $S(z)$ lead to a direct routine for computing all of their entries.   In particular, we can infer each row/column from entries in the preceding row/column:
\begin{corollary}\label{Corollary:RowColRec}
The rows and columns of $L(z)$ satisfy a three-term recurrence relationship:
\meeq{
 \I L(z) \bfe_{j-1}  + (2j+1) C L(z) \bfe_j - \I L(z) \bfe_{j+1}  = (2j+1) F(z) \bfe_j, \ccr
-L(z)^\top \bfe_{k-1}   - \I (2k+1)  C L(z)^\top \bfe_k  +  L(z)^\top \bfe_{k+1}  = (2k+1) F(z)^\top \bfe_k.
}
\end{corollary}

We consider computing the truncation up to total degree $p$:
\[
L^p(z) := \begin{pmatrix}
L_{00}(z) & \cdots & L_{0,p-1}(z) & L_{0p}(z) \\
L_{10}(z) & \cdots & L_{1,p-1}(z) \\
\vdots & \iddots \\
L_{p0}(z)
\end{pmatrix}
\]
Note that \corref{RowColRec} allows us to deduce the $(j+1)$-th column ($(k+1)$-th row)) of $L^p(z)$ from the $j$th column ($k$th row).
 We use this in Algorithm~\ref{Algorithm:Newtonian} to alternatively filling in a row and column
until we have computed all the entries of $L^p(z)$. For the right-hand side we denote
\[
F^{(\alpha),p}(z) := \begin{pmatrix}
F^{(\alpha)}_{00}(z) & \cdots & F^{(\alpha)}_{0,p-1}(z) & F^{(\alpha)}_{0p}(z) \\
F^{(\alpha)}_{10}(z) & \cdots & F^{(\alpha)}_{1,p-1}(z) \\
\vdots & \iddots \\
F^{(\alpha)}_{p0}(z)
\end{pmatrix}.
\]

\begin{algorithm}
\caption{Computing complex logarithmic integrals for  Legendre polynomials up to total degree $p$, $L^p(z) \in \bbC^{(p+1) \times (p+1)}$}
\label{Algorithm:Newtonian}
\begin{algorithmic}
\State{
Construct $F^{(1),p}(z)$
using \thref{2d_legendre_recurrence}. Note that $M_0(z\pm1),\ldots,M_{p+1}(z\pm 1)$ can be computed via the recurrence
\corref{Mrecurrence}  in $O(p)$ operations.}
\State{Construct $F^{(2),p}(z)$ also using \thref{2d_legendre_recurrence}. Note that $L_0(z\pm\I),\ldots,L_{p+1}(z\pm \I)$ can be computed via the recurrence in
\corref{Mrecurrence}, whilst  $C_2^{(-3/2)}(x),\ldots,C_{p+2}^{(-3/2)}(x)$ and $C_1^{(-1/2)}(y),\ldots,C_{p+1}^{(-1/2)}(y)$ can be computed via \eqref{eq:ultraspherical_three_term}.}
\State{Compute the first row/column of $L^p(z)$ using the three-term recurrences as defined in \propref{00entry} and \corref{Lfirstrowcol}.}

\For{$\ell =0,\ldots,\lfloor {p\over 2} \rfloor -1$}
\State{Compute $L_{\ell+1,\ell+1}(z),\ldots, L_{p-\ell-1,\ell+1}(z)$ from $L^p(z) \bfe_\ell$ and $L^p(z) \bfe_{\ell-1}$ (for $\ell \geq 1$) using the first equation in \corref{RowColRec}. }
\State{Compute $L_{\ell+1,\ell+2}(z),\ldots, L_{\ell+1,p-\ell-1}(z)$ from $ \bfe_\ell^\top L^p(z)$ and $\bfe_{\ell-1}^\top L^p(z)$ (for $\ell \geq 1$) using the second equation in \corref{RowColRec}. }
\EndFor
\State \Return $L^p(z)$.
\end{algorithmic}
\end{algorithm}

Algorithm~\ref{Algorithm:Newtonian} is effective for computing Newtonian potentials of   Legendre polynomials, by taking the real part of $L^p(z)$. As this algorithm computes the values exactly  the only source of error is round-off error; there are no errors introduced by numerical approximation. For low order $p$ the round-off error remains small when $z$ is on or near the square, see   \figref{lowperror}. This implies these techniques can be effectively combined with more traditional quadrature-based approach for $z$ bounded away from the square, where $h$-refinement is used to ensure only low $p$ is needed to resolve the desired function.  The round-off error grows exponentially with $p$ at an increasing rate the further away $z$ is from the square, see \figref{increasingperror}. Using high-precision arithmetic is one approach to mitigating this issue and we consider double-word arithmetic (cf.~\cite{joldes2017tight}) as implemented in the {\tt Double64} type in \cite{DoubleFloats}. Double-word arithmetic is based on using two traditional floating point numbers to capture 85 bits of precision (compared with the 52 bits of standard double precision), and is very performant.

\begin{figure}[tb]
\includegraphics[width=0.48\textwidth]{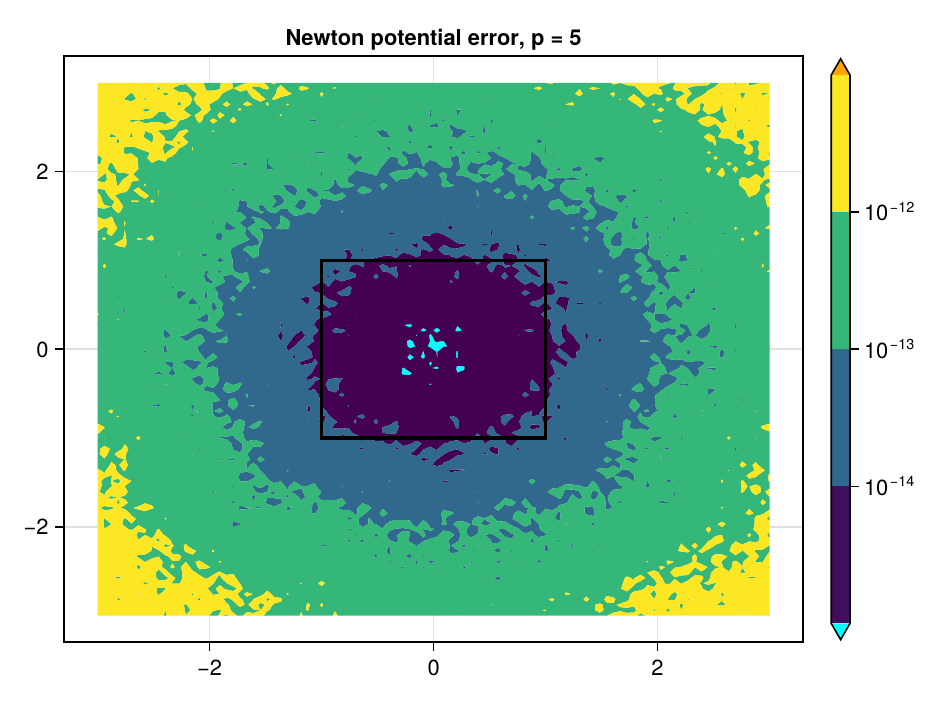}
\includegraphics[width=0.48\textwidth]{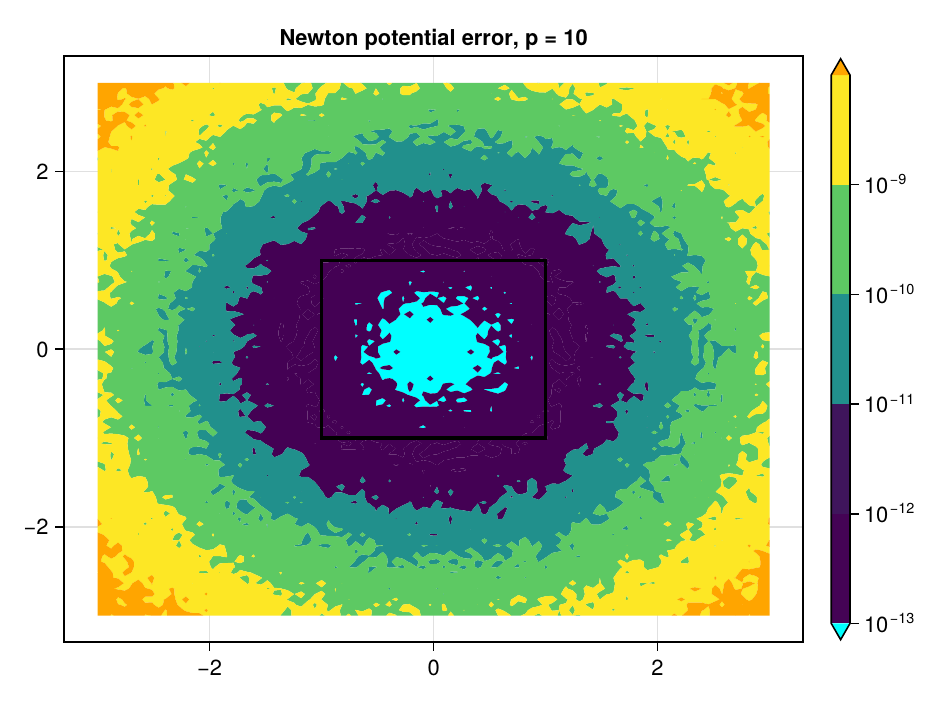}
\caption{The norm of the round-off error in computing the Newtonian potential for $p=5$ (left) and $p=10$ (right) using Algorithm~\ref{Algorithm:Newtonian} to compute the
$L^p(z)$. We see that the approximation is accurate to about 12–15 digits within the square, remains accurate in a neighbourhood of the square but
loses accuracy away from the square, where traditional quadrature becomes effective.} \label{Figure:lowperror}
\end{figure}

\begin{figure}[tb]
\centering \includegraphics[width=0.48\textwidth]{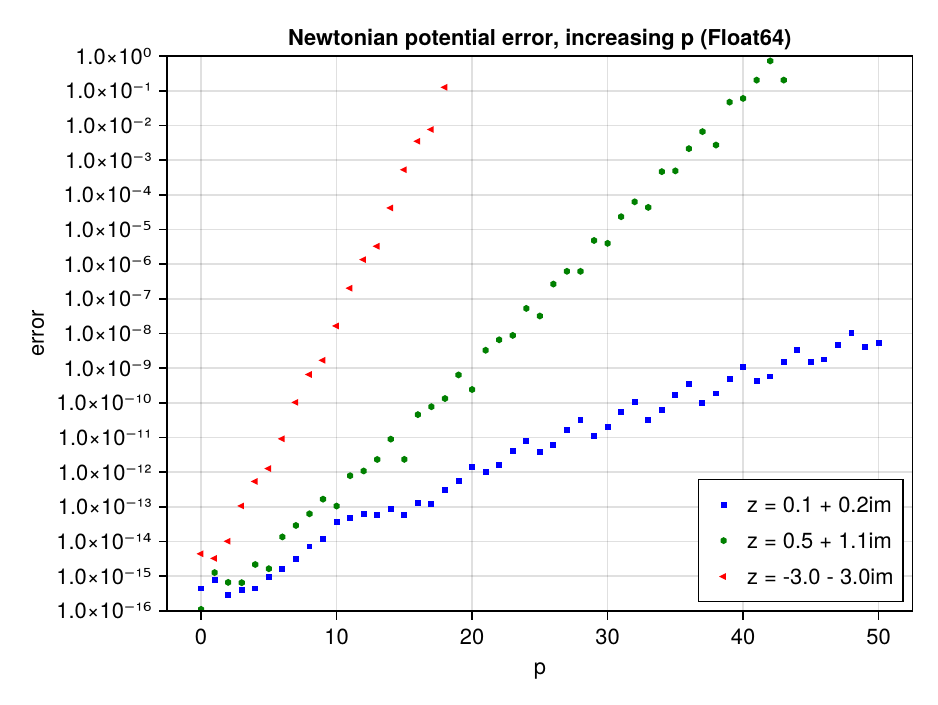}
\centering \includegraphics[width=0.48\textwidth]{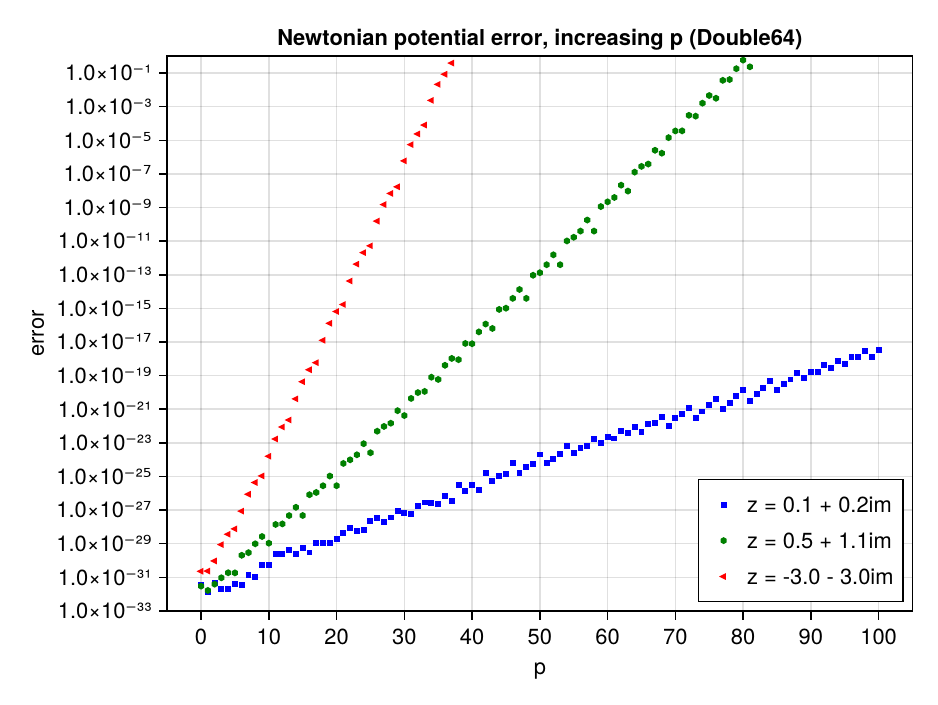}
\caption{The worst-case round-off error in computing the Newtonian potential for increasing $p$ using Algorithm~\ref{Algorithm:Newtonian} for three different points.  The error grows significantly with $p$, but at a much slower rate for $z$ near or within the square, which is the regime that traditional quadrature breaks down. Using double-word ({\tt Double64}) arithmetic increases the accuracy and allows for machine-precision accuracy inside the square up to around $p = 100$.} \label{Figure:increasingperror}
\end{figure}

 The recurrence is  fast to solve as demonstrated in \figref{logspeed}\footnote{The computations were performed on an M2 MacBook Air with a single thread.}. As a rough comparison we compare the cost with that of calculating  $\log(s+{\rm i} t)$ at $(p+1)^2$ points, which is a proxy for the cost of quadrature-based approaches which require kernel evaluation. Putting aside the round-off errors, the recurrence is faster than the cost of evaluating complex logarithms as it has significantly fewer special function operations: most of the computation involves floating point arithmetic operations. While with double-word precision as implemented in \cite{DoubleFloats} the computational cost is greater than evaluating the complex logarithm it is still in the same ballpark, and thus is potentially competitive for high-accuracy computations when compared to quadrature-based approaches.

\begin{figure}[tb]
\centering \includegraphics[width=0.6\textwidth]{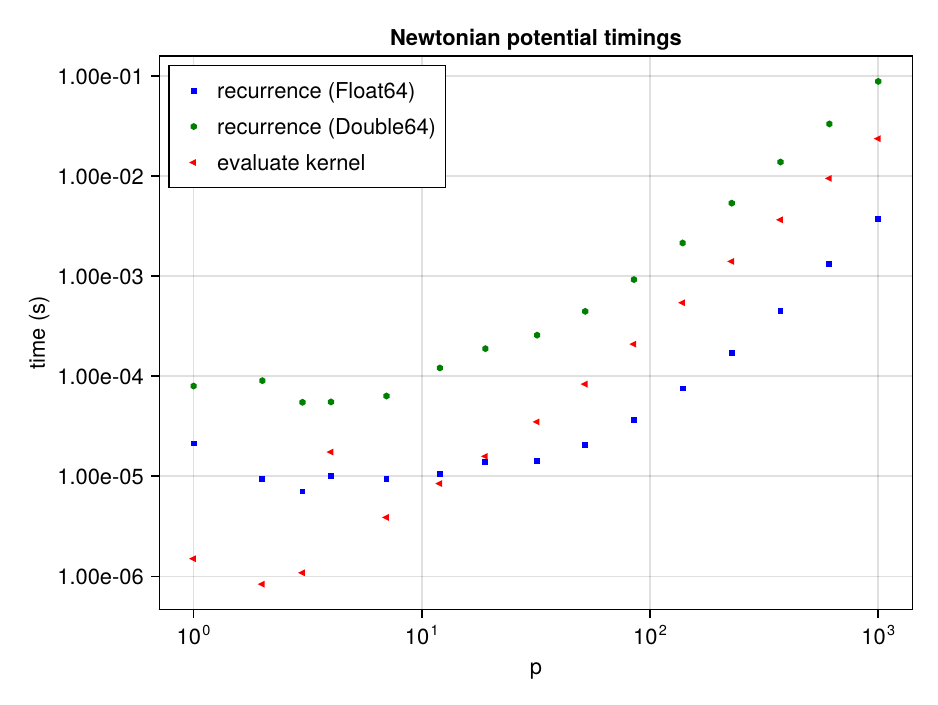}
\caption{Timing for computing the recurrence with double precision ({\tt Float64}) and double-word precision ({\tt Double64}), compared with the cost of evaluating $\log(s+ {\rm i} t)$ at $(p+1)^2$ points ({\it evaluate kernel}), which we use as a proxy for quadrature-based approximations. The computational cost of Algorithm~\ref{Algorithm:Newtonian} is competitive, though the round-off error is currently a limiting factor.}\label{Figure:logspeed}
\end{figure}

We can similarly  compute the Stieltjes integral $S(z)$ 
using the following recurrence on its rows and columns:

\begin{corollary}\label{Corollary:StieltjesRowColRec}
The rows and columns of $S(z)$ satisfy a three-term recurrence relationship: for $k,j>0$ we have 
\meeq{
 \I j S(z) \bfe_{j-1} + (2j+1)(B-zI) S(z)\bfe_j   + \I  (j+1) S(z) \bfe_{j+1} = 0, \ccr
  k S(z)^\top \bfe_{k-1}  + (2k+1)(\I B - zI) S(z)^\top \bfe_k  +   (k+1) S(z)^\top \bfe_{k+1} = 0.
}
\end{corollary}

That is, we can compute
\[
S^p(z) := \begin{pmatrix}
S_{00}(z) & \cdots & S_{0,p-1}(z) & S_{0p}(z) \\
S_{10}(z) & \cdots & S_{1,p-1}(z) \\
\vdots & \iddots \\
S_{p0}(z)
\end{pmatrix}
\]
using Algorithm~\ref{Algorithm:Stieltjes}.

\begin{algorithm}
\caption{Computing the Stieltjes integrals for  Legendre polynomials up to total degree $p$, $S^p(z) \in \bbC^{(p+1) \times (p+1)}$}
\label{Algorithm:Stieltjes}
\begin{algorithmic}
\State{Compute the first row/column of $S^p(z)$ using \propref{StieltjesRowCol}.  Note that $M_0(z\pm1),\ldots,M_{p+1}(z\pm 1)$ and $M_0(-{\rm i}z\pm1),\ldots,M_{p+1}(-{\rm i}z\pm 1)$ can be computed via the recurrence
\corref{Mrecurrence}  in $O(p)$ operations.}
\For{$\ell =0,\ldots,\lfloor {p\over 2} \rfloor -1$}
\State{Compute $S_{\ell+1,\ell+1}(z),\ldots, S_{p-\ell-1,\ell+1}(z)$ from $S^p(z) \bfe_\ell$ and $S^p(z) \bfe_{\ell-1}$ (for $\ell \geq 1$) using the first equation in \corref{StieltjesRowColRec}. }
\State{Compute $S_{\ell+1,\ell+2}(z),\ldots, S_{\ell+1,p-\ell-1}(z)$ from $ \bfe_\ell^\top S^p(z)$ and $\bfe_{\ell-1}^\top S^p(z)$ (for $\ell \geq 1$) using the second equation in \corref{StieltjesRowColRec}. }
\EndFor
\State \Return $S^p(z)$.
\end{algorithmic}
\end{algorithm}

\begin{figure}[tb]
\centering \includegraphics[width=0.48\textwidth]{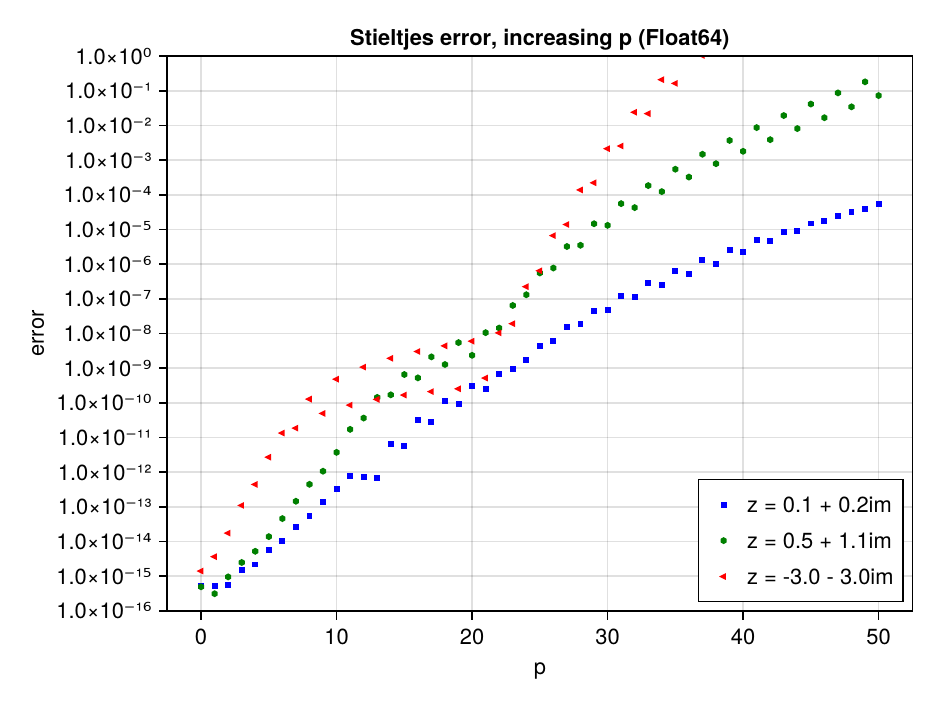}
\centering \includegraphics[width=0.48\textwidth]{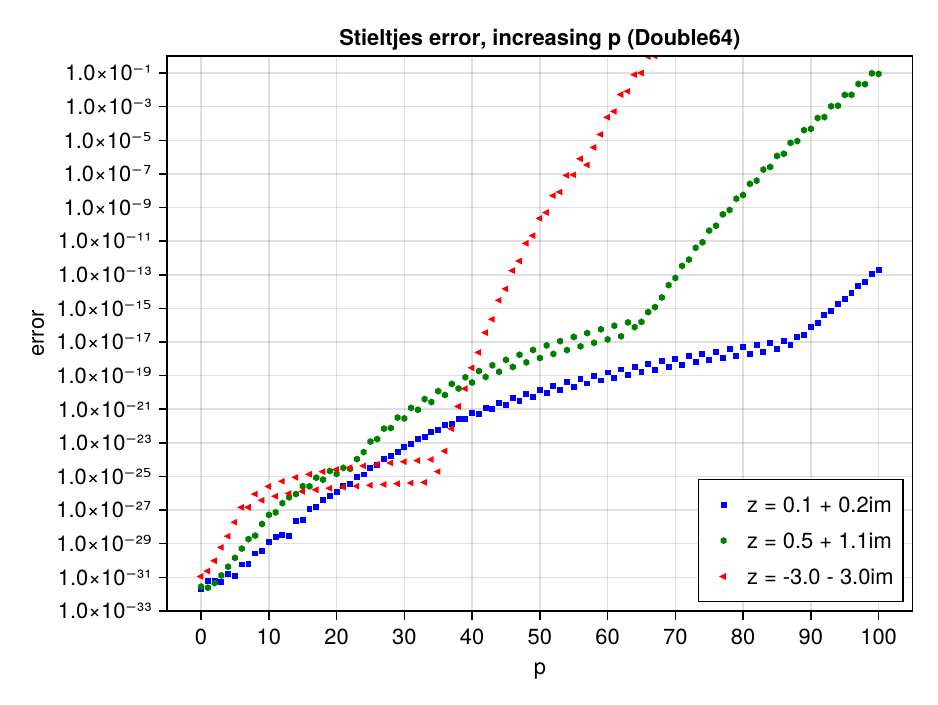}
\caption{The worst-case round-off error in computing the Stieltjes integrals for increasing $p$ using Algorithm~\ref{Algorithm:Stieltjes} for three different points.  The error behaviour is roughly similar to the Newtonian potential though with a kink phenomena dependent on the choice of $z$, likely due to the dependency on $z$ in the Sylvester equation. Using double-word ({\tt Double64}) arithmetic increases the accuracy and allows for machine-precision accuracy inside the square up to around $p = 90$.} \label{Figure:increasingpstieltjes}
\end{figure}

In \figref{increasingpstieltjes} we see that the error for the Stieltjes integral (whose real and imaginary part give the gradient of the Newtonian potential) behaves loosely similar to the Newtonian potential, though with larger errors including within the square. Double-word arithmetic continues to be effective for evaluating the recurrence with high-accuracy for moderately high $p$.

\section{Conclusion}
\label{Section:conclusion}

We have shown that complex logarithmic and Stieltjes integrals of  Legendre polynomials, whose real  and imaginary parts contain the Newtonian potential and its gradient, have displacement structure. This leeds to an extremely fast approach for evaluating Newtonian potentials on and near the square, though there are issues with round-off error for large orders.  

There are a number of directions for which this result may be useful in real-world computations. One may combine the recurrence-based approach for evaluating Newtonian potentials on and near the square with quadrature-based techniques away from the square, though the precise switch-over point requires further investigation. When combined with $h$-refinement this would result in an effective, fast, and parallelisable approach for computing Newtonian potentials. 

An arguably more interesting direction would be utilisation of the displacement structure for faster computations. The structure of our system is very similar to that of the Poisson equation as discretised in \cite{fortunato2020fast}, which was solved efficiently with the Alternative Direction Implicit method. There is also a recent algorithm which avoids inversion \cite{ballew2025akhiezer} which may also apply to our setting. A hiccup in utilising this result is that whilst our matrices have displacement structure, the right-hand side of the Sylvester equation does not decay rapidly for $z$ in the square hence truncation will result in large errors: we need to work with the recurrence in an infinite-dimensional way. 
Another alternative is to use mixed precision computation which was introduced for computing an operator arising in fractional calculus which had displacement structure in \cite{pu2023numerical}.  These potential approaches for stable computation would be significantly slower than the unstable direct recurrence, however.

Displacement structure also has implications in terms of utilising the matrices in further computation. It is known that we can perform fast matrix-vector operations \cite{heinig1984algebraic} and we further have precise bounds in the decay of the singular values  \cite{beckermann2017singular,beckermann2019bounds}. This latter result tells us that we in principle do not need to compute the full matrix, but rather can approximate it by a low-rank matrix which may lead to much more efficient evaluation of Newtonian potentials. 

In terms of generalisation of the result we suspect that complex logarithmic and Stieltjes integrals applied to orthogonal polynomials on triangles and other simple two-dimensional geometries  have displacement structure. It seems likely there is also displacement structure in higher dimensions,  however, the derivation based on complex analysis and Stieltjes integrals is unlikely to generalise. Recent results on computing power law integrals of orthogonal polynomials on intervals \cite{gutleb2022computing} and balls \cite{gutleb2023computation} may facilitate generalisation of displacement structure results, in particular to power law integrals on squares and cubes.

\appendix

\section*{Acknowledgments}
I would like to thank Thomas Anderson (Rice) who suggested this problem and discussed the work in depth. This work was supported by an EPSRC grant (EP/T022132/1).

\bibliographystyle{siamplain}
\bibliography{references}
\end{document}